\documentclass{article}
\usepackage{amssymb}
\newcommand{\bC}{\mathbb{C}}
\newcommand{\bP}{\mathbb{P}}

\def\ord{\mathop{\rm ord}}
\def\grad{\mathop{\rm grad}}
\def\lc{\mbox{\rm lc}}
\def\Zer{\mathop{\rm Zer}}
\def\ini{\mathop{\rm in}}
\newtheorem{Theorem}{Theorem}[section]
\newtheorem{Lemma}[Theorem]{Lemma}
\newtheorem{Corollary}[Theorem]{Corollary}
\newtheorem{Definition}[Theorem]{Definition}
\newtheorem{Example}[Theorem]{Example}
\newenvironment{proof}[1][Proof]{\textbf{#1.} }{\
\rule{0.5em}{0.5em}}

\begin{document}

\title{Ephraim's Pencils}
\author{Janusz Gwo\'zdziewicz}
\maketitle

\begin{abstract} 
Let $f(x,y)=0$ and $l(x,y)=0$ be respectively a singular and a regular analytic curve 
defined in the neighborhood of the origin of the complex plane. 
We study the family of analytic curves $f(x,y)-t\, l(x,y)^M=0$, where $t\neq 0$ is a 
complex parameter. For all but a finite number of parameters the curves of this family have 
the same embedded topological type. The exceptional parameters are called special values. 
We show that the number of special values does not exceed the number of components 
of the curve $f(x,y)=0$ counted without multiplicities. Then we apply this result to estimate 
the number of critical values at infinity of complex polynomials in two variables. 
\end{abstract}

\section{Main result}
Let $f(x,y)$, $g(x,y)\in\bC\{x,y\}$, $f(0,0)=g(0,0)=0$ be convergent power series. 
A family $f(x,y)-t\,g(x,y)=0$, where $t\in\bC$ is a parameter, is called a pencil of plane analytic curves. 
For all but a finite number of parameters the curves of the pencil have the same embedded topological type. 
We call $t_0$ a special value if the curve $f(x,y)-t_0g(x,y)=0$ is not topologically equivalent with 
a generic curve of a pencil. 

In this paper we consider pencils of the form 
$$ f(x,y)-t\,l(x,y)^M=0,
$$
where $l(x,y)=0$ is a smooth curve 
i.e. $l(x,y)=ax+by+\mbox{\textit{higher order terms}}$,  $ax+by\not\equiv0$ 
and $M$~is a positive integer.  
They will be called  Ephraim's pencils 
after Ephraim who used them in \cite{Ephraim} to study the singularities at infinity 
of plane algebraic curves  (in~\cite{Canbel} they are called Iomdin L\^e deformations). 
The main result of this article is 

\begin{Theorem}\label{Th:main}
Let $f(x,y)$, $l(x,y)\in\bC\{x,y\}$, $f(0,0)=l(0,0)=0$ be convergent power series 
without common factor. 
Assume that the curve $l(x,y)=0$ is smooth and that the curve
$f(x,y)=0$ has~$d$ components counted without multiplicities. 
Then the pencil $f(x,y)-t\,l(x,y)^M=0$, 
where $M$ is a positive integer, has at most $d$ nonzero special values. 
\end{Theorem}

\begin{Example}\label{Ex:1}
Let $H(Z)=(Z-a_1)\cdots(Z-a_d)$ be a complex polynomial with
$d-1$ critical values. Consider the pencil~$f_t(x,y)=0$ where 
$ f_t(x,y)=y^dH(x/y)-ty^d$.  A generic curve of this pencil consists of~$d$ 
straight lines crossing at the origin. 
The curve $f_{t_0}(x,y)=0$ is special if and only if 
$t_0$ is a critical value of the polynomial $H$; then $f_{t_0}(x,y)=0$ consists of $d-2$ single straight 
lines and one double line. 
\end{Example}

\begin{Example}\label{Ex:2}
Let $H(Z)=\prod_{i=1}^d(Z-a_i)^{m_i}$ be a complex polynomial
with $d-1$  nonzero  critical values different from $H(0)$. 
Consider the family of quasihomogeneous polynomials $ f_t(x,y)=\prod_{i=1}^d(x^2-a_iy^3)^{m_i}-ty^{3\deg H}$. 
For $t=0$ the curve $f_t(x,y)=0$ consists of $d$ multiple cusps $(x^2-a_iy^3)^{m_i}=0$.  
For $t=H(0)$ the curve $f_t(x,y)=0$ consists of $\deg H-1$ cusps and one double line $x^2=0$. 
If $t$ is a nonzero critical value of $H$ then the curve $f_t(x,y)=0$ has $\deg H-2$ single cusps and 
one double cusp.  For all other $t\in\bC$ the curve $f_t(x,y)=0$ decomposes into $\deg H$ 
pairwise different cusps of the form $x^2-b_iy^3=0$ for~$i=1,\dots,\deg H$. 
\end{Example}

This example shows that the bound for the number of special values 
given in Theorem~\ref{Th:main} is sharp.  

\section{Critical values at infinity} 
Let $f$ be a complex polynomial of two variables. 
The complex number $\tau$ is called a \textit{critical value at infinity} 
of the polynomial $f$ if there exists a sequence of points~$\{p_k\}\in\bC^2$ such that 
$|p_k|\to\infty$, $\grad f(p_k)\to 0$  and $f(p_k)\to \tau$ as~$k\to \infty$.  Several 
equivalent definitions are in \cite{Durfee}. 

Many authors described the set of critical values at infinity. 
L\^e and Oka in \cite{LVT-Oka} proposed the bound for their number 
in terms of the Newton diagram of $f$ and quasi-homogeneous polynomials associated with the 
edges of the Newton diagram.  Their bound was improved in \cite{L-M}. 

In \cite{G-P} is shown that every polynomial $f$ with finite set of critical points 
has at most $\min(1,\deg f-3)$ critical values at infinity. 

Below we give an estimation by the number of branches at infinity. 

\begin{Theorem}\label{Th:Crit}
Assume that the complex algebraic curve $f(x,y)=0$  has $n$ branches at infinity counted without multiplicities. 
Then the polynomial~$f$ has at most $n$ critical values at infinity different from $0$. 
\end{Theorem}

This result is in the spirit of Moh's theorem quoted below in Ephraim's version~\cite{Ephraim}.

\begin{Theorem}
Assume that the complex algebraic curve $f(x,y)=0$ has only one branch at infinity. 
Then $f$ has no critical values at infinity. In particular 
all curves $f(x,y)=\tau$ for $\tau\in \bC$ are equisingular at infinity. 
\end{Theorem}

\begin{Example}\label{Ex:3}
Let $H(x)=(x-a_1)\cdots(x-a_d)$ be a complex polynomial with
$d-1$ non-zero critical values and let $f(x,y)=H(x)$. 
As the curve $f(x,y)=0$ is the union of $d$ straight lines it has $d$ branches at infinity. 
The critical values at infinity of $f$ are exactly the critical values of $H$. 
\end{Example}

Next example is taken from \cite{G-S}.

\begin{Example}\label{Ex:4}
Let $A=x^k-1$, $B=x^{k+1}/(k+1)-x$ and $f(x,y)=(A^2y-x)^2+B$. 
The polynomial $f$ has no critical points in $\bC^2$ and has
$2k$ non-zero critical values at infinity. 
One can check that if $k$ is even then the curve $f(x,y)=0$ has 
$2k+1$ branches at infinity.  
\end{Example}

Example~\ref{Ex:4} shows that Theorem~\ref{Th:Crit} is almost optimal 
even in the class of polynomials without critical points in $\bC^2$. 

\medskip
\textbf{Question.}
Does there exists a polynomial $f(x,y)$ with $n$ non-zero critical values at infinity 
such that the curve $f(x,y)=0$ has $n$ branches at infinity?

\medskip
\begin{proof}[Proof of Theorem~\ref{Th:Crit}]
Let us embed $\bC^2$ into the projective plane $\bP^2$
equipped with homogeneous coordinates $X$, $Y$, $Z$ such that $x=X/Z$ and $y=Y/Z$. 
Then the line at infinity $L_{\infty}$ has an equation $Z=0$. 

Write 
$ f(x,y)=f_0+f_1(x,y)+\cdots+f_d(x,y)$
where $f_i(x,y)$ is a homogeneous polynomial of degree~$i$ for $i=0,\dots,d$. 
Since an affine change of coordinates does not affect neither critical values at infinity 
not the number of branches at infinity we may assume (applying it if necessary) that 
$f_d(x,y)$ is not divisible by $y$. 

Let $C_t$ be the projective closure of the curve $f(x,y)=t$. Then $C_t$ is given by the equation
$F(X,Y,Z)-tZ^d=0$, where $F(X,Y,Z)$ is the homogeneous polynomial 
defined by 
$$ F(X,Y,Z)=f(X/Z,Y/Z)Z^d=f_0Z^d+f_1(X,Y)Z^{d-1}+\cdots+f_d(X,Y)\;.
$$
The intersection of $C_t$ and the line at infinity is independent on $t\in\bC$ and 
is a finite set $\{(X;Y;Z)\in\bP^2 \,:\, Z=f_d(X,Y)=0\}$. 

By \cite{Durfee}  (see also \cite{LVT-Oka}) a complex number $\tau$ is a 
critical value at infinity of the polynomial $f$ if and only if there exist $P\in C_t\cap L_{\infty}$
such that the germ~$(C_{\tau},P)$  is a special curve of a family of germs of analytic curves~$(C_t,P)_{t\in\bC}$.
 
Take a point $A=(a;1;0)$ of  $C_t\cap L_{\infty}$ and consider a local analytic system of coordinates 
$u=\frac{X}{Y}-a$, $v=\frac{Z}{Y}$ centered at $A$. In these coordinates the curve $C_t$ has an 
equation $F(u+a,1,v)-tv^d=0$.  It follows that the family of curve germs $(C_t,A)_{t\in\bC}$ is an Ephraim's 
pencil. By Theorem~\ref{Th:main} the number of special curves of $(C_t,A)_{t\in\bC\setminus\{0\}}$ does not 
exceed the number of branches of the curve germ $(C_0,A)$ counted without multiplicities. 

Recall that branches of curve germs $(C_0,P)$, for $P\in C_t\cap L_{\infty}$ are branches
at infinity of $f(x,y)=0$. 
Running over all points of $C_t\cap L_{\infty}$ and using the above mentioned characterization of 
critical values at infinity we get Theorem~\ref{Th:Crit}.
\end{proof}

\section{Proofs}
In this part of the article we prove Theorem~\ref{Th:main} 
using the tree model of $f(x,y)$ introduced by Kuo and Lu in \cite{Kuo-Lu}.

\subsection{Notation}
By a fractional (convergent) power series we mean a series of the form  
$$ \phi(y)=a_1y^{n_1/N}+a_2y^{n_2/N}+\cdots,
   \quad a_i\in\bC,
$$
where $n_1<n_2<\dots$ are positive integers, such that $\phi(t^N)$ has positive 
radius of convergence. If $a_1\neq0$ then by definition the initial part of $\phi$ is 
$\ini\phi=a_1y^{n_1/N}$ and the order of $\phi$ is $\ord\phi=n_1/N$.
By convention the order of the zero series is $+\infty$.
For every fractional power series $\phi(y)$, $\psi(y)$ denote
$O(\phi,\psi)=\ord(\phi(y)-\psi(y))$ and call this number the contact order. 

Let $p(x,y)\in\bC\{x,y\}$ be a convergent power series. 
A fractional power series~$x=\gamma(y)$ is called a Newton-Puiseux root of $p(x,y)$ 
if $p(\gamma(y),y)=0$. 
We denote $\Zer p$ the set of all Newton-Puiseux roots of $p(x,y)$. 

Hereafter ``$+\cdots$'' will mean ``plus higher order terms''.

\subsection{Special values of a pencil}
\begin{Theorem}\label{T2}
Let $f(x,y)$ be a convergent power series  such that $0<\ord f(x,0)<\infty$ 
                and let $M$ be a positive integer. Then the following conditions 
                are equivalent:
\begin{description}
	\item[(i)] $t_0$ is a special value of the pencil $f(x,y)-ty^M=0$, 
	\item[(ii)] the image of the curve $f_x'(x,y)=0$ by the mapping 
	      $\Phi:(\bC^2,0)\to(\bC^2,0)$ given by $(u,v)=(y^M,f(x,y))$ has 
	      a tangent $v-t_0u=0$,
	\item[(iii)] there exists a Newton-Puiseux root $\beta$ of $f_x'(x,y)$ 
	      such that $f(\beta(y),y)=t_0 y^M+\cdots\;$.
 \end{description}
\end{Theorem}                

\begin{proof}
The equivalence \textbf{(i)}~$\Leftrightarrow$~\textbf{(ii)} 
is a direct consequence of \cite{Casas2}, Theorem~5.1 because 
curves $f(x,y)-ty^M=0$ are inverse images of straight lines 
$v-tu=0$ by $\Phi$. 

In order to prove \textbf{(ii)}~$\Leftrightarrow$~\textbf{(iii)} we will use 
the property that every branch, that is  an irreducible curve germ, has only one tangent.
If a branch $h(u,v)=0$ has an analytic parametrization $s \to (s^k, as^k+\cdots)$
then the equation of the tangent is $v-au=0$.

For every branch $\lambda$ of the curve $f_x'(x,y)=0$ there exists 
a Newton-Puiseux root $x=\beta(y)$ of $f_x'(x,y)$ and a positive integer $n$
such that $x=\beta(s^n)$, $y=s^n$ is an analytic parametrization 
of $\lambda$ and conversely for every Newton-Puiseux root of $f_x'(x,y)$ there exists a
branch~$\lambda$ with such a parametrization. 

Take $x=\beta(y)$ and the associated branch $\lambda$ Then the image of $\lambda$ by $\Phi$ 
is also the branch given by $u=s^{Mn}$, $v=f(\beta(s^n),s^n)$. 
This branch has the tangent $v-t_0u=0$ if and only if~ 
$f(\beta(s^n),s^n)=t_0s^{Mn}+\cdots\,$. This gives \textbf{(ii)}~$\Leftrightarrow$~\textbf{(iii)}.
\end{proof}

\subsection{Pseudo-balls}
\begin{Definition} 
Let $\alpha$ be a fractional power series and 
let $h$ be a positive rational number or $+\infty$. 
By definition the pseudo-ball $B(\alpha,h)$ is the set 
$$ B(\alpha,h)=\{\,\gamma \mbox{ is a fractional power series }: O(\gamma,\alpha)\geq h \,\} .
$$
We call $h(B):=h$ the height of $B=B(\alpha,h)$.
\end{Definition}

Take a pseudo-ball $B=B(\alpha,h)$ with finite height. 
Let $\lambda_B(y)$ denote $\alpha(y)$ with all terms $y^e$, $e\geq h$ omitted. 
Then every $\gamma\in B$ has a form 
$$ \gamma(y)=\lambda_B(y)+cy^h+\cdots \;.
$$
We call the number $c$ \textit{the leading coefficient of $\gamma$ with respect to~$B$} 
and denote it $\lc_B(\gamma)$.

\subsection{Newton-Puiseux roots of a partial derivative}
Let $f(x,y)$ be a convergent power series such that $\ord f(x,0)=p>1$. 
Then the Newton-Puiseux factorizations of $f$ and $f_x'$ are of the form 
\begin{eqnarray*}
f(x,y)    &=& u(x,y)\prod_{i=1}^p [x-\alpha_i(y)] \\
f_x'(x,y) &=& u'(x,y)\prod_{j=1}^{p-1} [x-\beta_j(y)] 
\end{eqnarray*}
where $u$, $u'$ are units and $\alpha_i$, $\beta_j$ 
are fractional power series. 

\begin{Lemma}\label{L2}
Let $B$ be a pseudoball of finite height. 
Then there exists a complex polynomial $F_B(z)$ and an exponent $q(B)$ 
such that for every $\gamma(y)=\lambda_B(y)+zy^{h(B)}+\cdots$ we have
\begin{equation}\label{Eq4}
f(\gamma(y),y)=F_B(z)y^{q(B)}+\cdots\; .
\end{equation}
Moreover 
\begin{equation}\label{Eq10}
F_B(z)=C\!\!\prod_{i:\alpha_i\in B} (z-\lc_B(\alpha_i))
\end{equation}where $C$ is a nonzero constant.
\end{Lemma}

\begin{proof}
An easy computation
\begin{eqnarray*}
\lefteqn{f(\gamma(y),y) = u(\gamma(y),y)\prod_{i=1}^p [\gamma(y)-\alpha_i(y)]  } \\ 
  & & = (u(0,0)+\cdots) \prod_{i:\alpha_i\notin B} [\ini(\lambda_B(y)-\alpha_i(y))+\cdots]
                     \prod_{i:\alpha_i\in B}[(z-\lc_B(\alpha_i))y^{h(B)}+\cdots]   \\
  & & = u(0,0) \prod_{i:\alpha_i\notin B} \ini(\lambda_B(y)-\alpha_i(y))
      \prod_{i:\alpha_i\in B} y^{h(B)}\prod_{i:\alpha_i\in B} (z-\lc_B(\alpha_i)) + \cdots
\end{eqnarray*}
shows that it is enough to take $C$ and $q(B)$ such that 
$$Cy^{q(B)}= u(0,0) \prod_{i:\alpha_i\notin B}\ini(\lambda_B(y)-\alpha_i(y)) 
                      \prod_{i:\alpha_i\in B}y^{h(B)}.$$
\end{proof}

\begin{Lemma}\label{L1}  Let $B=B(\alpha_k,h)$, $1\leq k\leq p$ 
be a pseudo-ball of finite height. Then
$$
F_B'(z) = 
\mbox{const} \prod_{j:\beta_j\in B} (z-\lc_B(\beta_j))  .
$$
\end{Lemma}

\begin{proof} First consider the special case when 
$h$ is an integer and $\alpha_i$, $\beta_j$ are convergent 
power series.
Let $\tilde f(x,y)=f(x-\lambda_B(y),y)$. 
Since $x\mapsto x-\lambda_B(y)$ is an analytic substitution  
$\frac{\partial}{\partial x}\tilde f(x,y) = 
\frac{\partial f}{\partial x}(x-\lambda_B(y),y)$. 
Hence, 
$\tilde f(x,y)=\tilde u(x,y)\prod_{i=1}^p [x-\tilde\alpha_i(y)]$, 
$\tilde f_x'(x,y)=\tilde u'(x,y)\prod_{j=1}^{p-1}[x-\tilde\beta_j(y)]$, 
where $\tilde u$, $\tilde u'$ are units 
and $\tilde\alpha_i(y)=\alpha_i(y)-\lambda_B(y)$,
$\tilde\beta_j(y)=\beta_j(y)-\lambda_B(y)$ for all $i$, $j$.
Let $\ord_w$ be a weighted order such that $\ord_wx=h$ and $\ord_wy=1$. 
Since the initial part of a product is the product of initial parts,
the weighted initial part of $\tilde f(x,y)$ is the polynomial 
\begin{equation}\label{Eq1}
\ini{}_{\!w}\tilde f(x,y) = C y^k\prod_{i:\ord\tilde\alpha_i\geq h}[x-\lc_B(\alpha_i)y^h]\, ,
\end{equation}
where $k$ is an integer and $C$ is a nonzero constant. Similarly
\begin{equation}\label{Eq2}
\ini{}_{\!w}\tilde f_x'(x,y) = C' y^l\prod_{j:\ord\tilde\beta_j\geq h}[x-\lc_B(\beta_j)y^h]\, .
\end{equation}
One easily checks that 
$\frac{\partial}{\partial x}(\ini_w\tilde f) = 
 \ini_w\frac{\partial\tilde f}{\partial x}$.
Hence, by~(\ref{Eq1}) and (\ref{Eq2}) we get
$$
\frac{\partial}{\partial x} \prod_{i:\ord\tilde\alpha_i\geq h}[x-\lc_B(\alpha_i)y^h] = 
\mbox{\textit{const}}\!\!\!\!\!\prod_{j:\ord\tilde\beta_j\geq h} [x-\lc_B(\beta_j)y^h]
$$ 
Substituting in above equality $y=1$ and using the condition: 
$\gamma\in B \Leftrightarrow \ord (\gamma-\lambda_B)\geq h$
we get the statement of the lemma. 

Assume now that $h$ is an arbitrary positive rational number and $\alpha_i$,
$\beta_j$ are fractional power series. Then one can always find an integer~$D>0$ 
such that $Dh$ is an integer and $\alpha_i(y^D)$, $\beta_j(y^D)$ 
are convergent power series for all $i$, $j$. 
The proof reduces to the special case by taking 
$f(x,y^D)$ and a pseudo-ball $\bar B=B(\alpha_k(y^D),Dh)$ because 
$\gamma(y)\in B\Leftrightarrow \gamma(y^D)\in\bar B$ and 
$\lc_B(\gamma)=\lc_{\bar B}(\gamma(y^D))$ for $\gamma\in B$.
\end{proof}

\begin{Corollary}\label{Wn2}  
For every $\beta_j$ there exist $\alpha_k$, $\alpha_l$ such that 
$$ O(\alpha_k,\beta_j) = O(\alpha_l,\beta_j) = 
   O(\alpha_k,\alpha_l)= \max_{i=1}^p\, O(\alpha_i,\beta_j) .
$$
\end{Corollary}

Corollary~\ref{Wn2} is a part of  Lemma~3.3 in~\cite{Kuo-Lu}. Here we give an independent proof. 

\begin{proof} 
Take $\alpha_k$ such that $O(\alpha_k,\beta_j)=\max_{i=1}^p\, O(\alpha_i,\beta_j)$.
Let $h=O(\alpha_k,\beta_j)$ and $B=B(\alpha_k,h)$. 

Assume first that $h$ is finite.  By Lemma~\ref{L1}  $F_B'(\lc_B(\beta_j))=0$.
Since $h$ is the maximal contact order between Newton-Puiseux roots of $f(x,y)$ and $\beta_j$, we get
$\lc_B(\alpha_i)\neq\lc_B(\beta_j)$ for all $\alpha_i\in B$. It follows that $F_B(\lc_B(\beta_j))\neq0$.
  
Every non-constant complex polynomial which has a nonzero critical value has at least two different roots. 
Thus there exist $\alpha_l\in B$  such that $\lc_B(\alpha_l)\neq\lc_B(\alpha_k)$. 
We get $O(\alpha_k,\beta_j) = O(\alpha_l,\beta_j)=O(\alpha_k,\alpha_l)=h$.

If $h=+\infty$ then it is enough to take $\alpha_l=\alpha_k$.
\end{proof}

\subsection{Tree-model $T(f)$}
Consider the finite set 
${\cal B} := \{\,B(\alpha_i,O(\alpha_i,\alpha_j)): \alpha_i, \alpha_j\in \Zer f\,\}$ of pseudo-balls.
The inclusion relation gives $\cal B$ a structure of a tree
called \textit{the Kuo-Lu tree-model}~$T(f)$ (compare \cite{Kuo-Lu}).
The root of $T(f)$ is the pseudo-ball of the minimal height
which contains all Puiseux roots of $f(x,y)$.
Leaves of $T(f)$ are pseudo-balls
$B(\alpha_i,+\infty)=\{\alpha_i\}$ of infinite heights. 
A path from the root to the leave $\{\alpha_i\}$ connects succesive 
$B\in T(f)$ of increasing heights for which $\alpha_i\in B$.  

As in every finite tree we have a parent-child relation between elements of $T(f)$. 
Let $B'$ be a child of $B\in T(f)$. Since $h(B')>h(B)$, every element of~$B'$ has a form 
$\lambda_B(y)+cy^{h(B)}+\cdots$ with fixed  $c\in\bC$. 
As in  \cite{Kuo-Pa} we say that $B'$ is supported by $c$ and write $B\perp_c B'$.
One easily checks that for every element $c$ of the set $\{\, \lc_B(\alpha_k):\alpha_k\in B\,\}$ 
exists exactly one child of $B$ supported by $c$ of height 
$\min\{O(\alpha_k,\alpha_l): \alpha_k,\alpha_l\in B,\, \lc_B(\alpha_k)=\lc_B(\alpha_l)=c\}$. 

In order to simplify the statement of the next lemma we put  $q(B):=\infty$ if $B\in T(f)$ has infinite height. 

\begin{Lemma}\label{L3}
Let $B\perp_c B'$. Then $q(B)<q(B')$.
\end{Lemma}
 
\begin{proof} 
Assume that $B'$ has a finite height. and take $\gamma(y)=\lambda_{B'}(y)+c_1y^{h(B')}$ such that 
$F_{B'}(c_1)\neq0$.  
By Lemma~\ref{L2}  $f(\gamma(y),y)=F_{B}(c)y^{q(B)}+\cdots=F_{B'}(c_1)y^{q(B')}+\cdots$. 
Since $F_{B}(c)=0$, we get $q(B)<q(B')$. 
%
\end{proof}

\medskip

It follows from Corollary~\ref{Wn2}  that for every $\beta_k\in\Zer f_x'$ there exists 
exactly one $B\in T(f)$ such that $\beta_k\in B$ and  $h(B)=\max_{i=1}^p\, O(\alpha_i,\beta_k)$.
We will say after \cite{Kuo-Lu} that $\beta_k$ leaves the tree $T(f)$ at $B$.
Then $F_B(\lc_B(\beta_k))\neq0$ and by Lemma~\ref{L1} we have $F_B'(\lc_B(\beta_k))=0$.
\medskip

\begin{Example}
\label{Ex1}
Let $f(x,y)=x(x^2-3y^6)(x-4y^2)$. 
The Newton-Puiseux roots of $f(x,y)$ are: $\alpha_1=0$, $\alpha_2=\sqrt{3}y^3$, 
$\alpha_3=-\sqrt{3}y^3$ and $\alpha_4=4y^2$.

\noindent Let us draw the Kuo-Lu tree of $f$.  Following
\cite{Kuo-Lu} we draw pseudo-balls of finite height as horizontal
bars and we do not draw pseudo-balls of infinite height. The 
tree $T(f)$ has two bars (pseudo-balls) of finite hight: $B_1$ of height $2$ and $B_2$ of height~$3$. 

We have
$F_{B_1}(z)=\prod_{i=1}^4 (z-\lc_{B_1}(\alpha_i))=z^3(z-4)$. 
Since $F_{B_1}'(z)=4z^2(z-3)$ there is only one  
root $\beta_1=3y^2+\cdots$ of $f_x'(x,y)$ which leaves $T(f)$ at $B_1$. 
 
We have
$F_{B_2}(z)=-4\prod_{i=1}^3 (z-\lc_{B_2}(\alpha_i))=-4z(z^2-3)$. 
Since $F_{B_2}'(z)=-12(z-1)(z+1)$ there are two  
roots $\beta_2=y^3+\cdots$ and  $\beta_3=-y^3+\cdots$ of $f_x'(x,y)$ which leave $T(f)$ at $B_2$.

$$
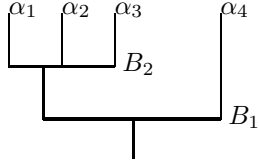
\begin{picture}(90,60)(0,0)
{\thicklines \put(13,15){\line(1,0){67}}}
{\thicklines \put(0,35){\line(1,0){40}}}
\put(13,15){\line(0,1){20}}  
\put(0,35){\line(0,1){20}} 
\put(20,35){\line(0,1){20}}
\put(40,35){\line(0,1){20}} 
\put(0,55){$\alpha_1$} 
\put(20,55){$\alpha_2$}
\put(40,55){$\alpha_3$}
\put(80,15){\line(0,1){40}} 
\put(47,0){\line(0,1){15}}  
\put(43,33){$B_2$}
\put(83,13){$B_1$}
\put(80,55){$\alpha_4$} 
\end{picture}
$$
\end{Example}

\begin{Theorem} \label{T3}
Let $f(x,y)$ be a convergent power series  
such that $0<\ord f(x,0)<\infty$  and let $M$ be a positive integer. 
Then $t_0\neq 0$ is a special value of the pencil $f(x,y)-ty^M=0$
if and only if there exists $B\in T(f)$ of finite height for which
$q(B)=M$ and $t_0$ is a critical value of the polynomial $F_B$.
\end{Theorem}                

\begin{proof} 
Assume that $t_0\neq0$ is a special value of the pencil $f(x,y)-ty^M=0$. 
Then by Theorem~\ref{T2} there exists $\beta\in\Zer f_x'$ such that 
$f(\beta(y),y)=t_0y^M+\cdots\,$.  
Let $B\in T(f)$ be the bar at which $\beta$ leaves the tree $T(f)$. 
Then $F_B'(\lc_B(\beta))=0$ and $F_B(\lc_B(\beta))\neq0$.
By Lemma~\ref{L2} we have $f(\beta(y),y)=F_B(\lc_B(\beta))y^{q(B)}+\cdots\,$. 
Hence $q(B)=M$ and $t_0=F_B(\lc_B(\beta))$ is a critical value of $F_B$. 

Conversly, if there exists $B\in T(f)$ with $q(B)=M$ such that 
$t_0$ is a critical value of the polynomial $F_B$ 
then by Lemma~\ref{L1} there exists 
$\beta\in\Zer f_x'\cap B$ such that $t_0=F_B(\lc_B(\beta))$. 
We get $f(\beta(y),y)=F_B(\lc_B(\beta))y^{q(B)}+\cdots\,$. 
Hence by Theorem~\ref{T2}, $t_0$ is a special value of the pencil $f(x,y)-ty^M=0$.
\end{proof}

\subsection{Action of roots of unity on fractional power series}
Let $D$ be a positive integer such that all Newton-Puiseux roots of $f(x,y)$
and $f_x'(x,y)$ can be written in the form 
\begin{equation}\label{Eq:ar}
 \phi(y)=a_1y^{n_1/D}+a_2y^{n_2/D}+\cdots, \qquad
 0\leq n_1<n_2<\dots\; .
\end{equation}

Let $\theta$ be any $D$-th complex root of unity, $\theta^D=1$. 
Each $\theta$ yields a transformation (conjugation) 
on series of form~(\ref{Eq:ar})
$$ \theta(\phi)(y)=a_1\theta^{n_1}y^{n_1/D}+a_2\theta^{n_2}y^{n_2/D}+\cdots\;.
$$
One easily verifies that $\theta$ acts trivially on the subring 
$\bC\{y\}\subset\bC\{y^{1/D}\}$
and that the action of~$\theta$ preserves a contact order, that is
$O(\phi_1,\phi_2)=O(\theta(\phi_1),\theta(\phi_2))$. 

The conjugate action permutes transitively the Newton-Puiseux roots of any irreducible 
$p(x,y)\in\bC\{x,y\}$ provided that $\Zer p\subset\bC\{y^{1/D}\}$ (see \cite{Walker}, page~107)

\begin{Definition}(\cite{Kuo-Pa}, Definition~6.1)
Take $B,\bar B\in T(f)$. We say $B$ is conjugate to $\bar B$ 
writing it as $B\sim\bar B$, if, and only if $h(B)=h(\bar B)$ 
and there exists an irreducible $p(x,y)\in\bC\{x,y\}$ 
of which one (Newton-Puiseux) root belongs to $B$ and one belongs to $\bar B$.
\end{Definition}

\begin{Lemma}\label{L:KP}(\cite{Kuo-Pa}, Lemma~6.2)
Suppose $B\sim\bar B$. Take any irreducible $q(x,y)\in\bC\{x,y\}$. If $q(x,y)$ 
has one root which belongs to $B$ then it also has a root which belongs to $\bar B$.
\end{Lemma}

It follows from Lemma~\ref{L:KP} that $\sim$ is an equivalence relation. One can 
use any irreducible component of $f(x,y)$ as $q(x,y)$ to identify an equivalence class
of any pseudo-ball of $T(f)$ at any given height. 

Take $B,\bar B\in T(f)$ such that $B\sim\bar B$. By Lemma~\ref{L:KP} there exists $\theta$, 
$\theta^D=1$ such that $\theta(\lambda_B)=\lambda_{\bar B}$. 
If $h(B)=\frac{n}{D}$, $q(B)=\frac{m}{D}$ then taking the conjugate of the
equation
$$ f(\lambda_B(y)+cy^{h(B)},y)=F_B(c)y^{q(B)}+\cdots
$$ 
by $\theta$ we get 
$$
f(\lambda_{\bar B}(y)+\theta^{n}cy^{h(\bar B)},y)=\theta^{m}F_B(c)y^{q(B)}+\cdots .
$$
It follows from~Lemma~\ref{L2} that 
\begin{equation}\label{q(B)}
q(\bar B)=q(B) \quad\mbox{for $\bar B\sim B$}
\end{equation}
and also follows that $F_{\bar B}(\theta^{n}z)=\theta^{m}F_B(z)$. 
In particular 
\begin{equation}\label{Eq:F}
F_{\bar B}(\theta^{n}z)=F_B(z) \quad\mbox{if $\bar B\sim B$ and $q(B)$ is an integer.}
\end{equation}

\begin{Definition}
The index of a fractional power series $\phi(y)$ is by definition the minimal positive integer $N$ such that
$\phi(y)\in\bC\{y^{1/N}\}$.
\end{Definition}

\begin{Theorem}\label{Th:4}
Take $B\in T(f)$ such that $q(B)$ is an integer. 
Assume that $\lambda_B(y)$ is a series of index~$N$
and that $h(B)=\frac{m}{Nn}$ where ${\rm gcd} (m,n)=1$. Then
there exists a polynomial $G_B$ for which $F_B(z)=G_B(z^n)$.
\end{Theorem}

\begin{proof}
Let $\epsilon$ be a primitive $n$-th root of unity. 
We shall show that for every complex number~$z$ one has $F_B(z)=F_B(\epsilon z)$.

Since series $\lambda_B(y)+zy^{h(B)}$, $\lambda_B(y)+\epsilon zy^{h(B)}$ 
are conjugate, there exists $\theta$, $\theta^D=1$ such that 
$\theta(\lambda_B(y)+zy^{h(B)})=\lambda_B(y)+\epsilon zy^{h(B)}$.

By Lemma \ref{L2} 
\begin{equation}\label{Eq:7}
   f(\lambda_B(y)+zy^{h(B)},y)= F_B(z)y^{q(B)}+\cdots
\end{equation}
and 
\begin{equation}\label{Eq:8}
   f(\lambda_B(y)+\epsilon zy^{h(B)},y)= F_B(\epsilon z)y^{q(B)}+\cdots\; .
\end{equation}
The conjugate of~(\ref{Eq:7}) by $\theta$ is \begin{equation}\label{Eq:9}
 f(\lambda_B(y)+\epsilon zy^{h(B)},y)= F_B(z)y^{q(B)}+\cdots
\end{equation}

By (\ref{Eq:8}) and (\ref{Eq:9}) we get  $F_B(z)=F_B(\epsilon z)$.
 
Let $Az^k$ be any nonzero monomial which appears in $F_B(z)$. 
By equality $F_B(z)=F_B(\epsilon z)$
we get $Az^k=A\epsilon^kz^k$ which gives $\epsilon^k=1$. 
It follows that $k$ is the multiple of $n$ and as a consequence
there exists the polynomial $G_B$ such that $F_B(z)=G_B(z^n)$.
\end{proof}

\begin{Lemma}\label{L:sim} 
Keep the assumtions and notations of Theorem~\ref{Th:4} and let 
$B\perp_{c_1} B_1$, $B\perp_{c_2} B_2$. Then $B_1\sim B_2$ if and 
only if $c_1^n=c_2^n$.
\end{Lemma}

\begin{proof} 
 Take a Newton-Puiseux root 
$\alpha(y)=\lambda_B(y)+c_1y^{h(B)}+\cdots$ of $f$.

Assume that $B_1\sim B_2$.
All conjugates to $\alpha$ which belong to $B$ are of the form 
$\lambda_B(y)+\epsilon c_1y^{h(B)}+\cdots$, where $\epsilon^n=1$. 
By Lemma~\ref{L:KP} one of conjugates to $\alpha$ belongs to $B_2$. 
Hence $c_2=\epsilon c_1$ which gives $c_1^n=c_2^n$. 

Conversly, if $\alpha(y)=\lambda_B(y)+c_1y^{h(B)}+\cdots$
and $c_1^n=c_2^n$ then $c_2=\epsilon c_1$ for some  $n$-th root of unity $\epsilon$. 
Hence there exists a conjugate to $\alpha$ of the form  
$\lambda_B(y)+c_2y^{h(B)}+\cdots$\,. This gives $B_1\sim B_2$.
\end{proof}

\medskip
Denote $[B]$ the equivalence class of $B\in T(f)$ with respect to the conjugate relation.
Let $E(f)$ be the set of all equivalence classes. 
Then $E(f)$ has also a structure of a tree. 
Following \cite{Eggers}, \cite{GB}, \cite{Wall} we call $E(f)$ the Eggers tree of~$f$. 
Irreducible (possibly multiple) factors of $f$ are in one-to-one correspondence with 
the leaves of $E(f)$. 

Denote $t([B])$ the number of children of $[B]$ in the 
Eggers tree $E(f)$.

\begin{Lemma}\label{L:E} 
Keep the assumtions and notations of Theorem~\ref{Th:4} and let 
$G_B(z)=\mbox{\textit{const}}\prod_{i=1}^k(z-b_i)^{m_i}$ with 
$b_i\neq b_j$ for $i\neq j$.  Then $t([B])=k$.
\end{Lemma}

\begin{proof} 
By Lemma~\ref{L:sim} the children of $[B]$ in the Eggers tree $E(f)$ are 
in one-to-one correspondence with the elements of the set 
$\{\, (\lc_B(\alpha_i))^n : \alpha_i\in B\,\}$ which is by the equality $F_B(z)=G_B(z^n)$ 
the set $\{\,b_1,\dots, b_k\,\}$. 
\end{proof}

\begin{Corollary}\label{Wn:E}  
Take $B\in T(f)$ such that $q(B)$ is an integer.  
Then the polynomial~$F_B$  has at most $t([B])$ non-zero critical values. 
\end{Corollary}

\begin{proof}
By Theorem~\ref{Th:4} and Lemma~\ref{L:E}
$F_B(z)=G_B(z^n)$, 
where $G_B(z)=\mbox{\textit{const}}\prod_{i=1}^{t([B])}(z-b_i)^{m_i}$ with 
$b_i\neq b_j$ for $i\neq j$. 
Computing the derivative $F_B'(z)=nz^{n-1}G_B'(z^n)$ we see that 
every critical value of $F_B$ is a critical value of the polynomial $G_B$ or 
the number $F_B(0)$ (if $n>1$).  Hence, the polynomial $F_B$ has 
at most $t([B])$ non-zero critical values. 
\end{proof}

\subsection{Proof of Theorem~\ref{Th:main}}
If $\Phi: (\bC^2, 0)\to(\bC^2, 0)$ is a local analytic diffeomorphism 
and $f_1=f\circ\Phi$, $l_1=l\circ\Phi$, then the pencils 
$f(x,y)-tl(x,y)^M=0$, $f_1(x,y)-tl_1(x,y)^M=0$
have the same special values.
Hence, without loss of generality we may assume that $l(x,y)=y$ (take such a  $\Phi$ 
that $l\circ\Phi=y$). 

It follows from~(\ref{q(B)}) that for every $[B]\in E(f)$ the number
$q([B]):=q(B)$ for $B\in [B]$ is well defined.  Consider the set 
$E_M=\{\,[B]\in E(f): q([B])=M\,\}$. 

\medskip
\textbf{Claim~1.}  
The number of non-zero special values of the pencil $f(x,y)-ty^M=0$ 
 is bounded from above by $\sum_{[B]\in E_M} t([B])$.

\smallskip
\textit{Proof.}
By Theorem~\ref{T3} the set of non-zero
special values of the pencil $f(x,y)-ty^M=0$ is the union of the sets 
of non-zero critical values of polynomials $F_B$, where $B\in T(f)$, $q(B)=M$.  

If $B\sim \bar B$ and $q(B)=M$  then by~(\ref{Eq:F})
the polynomials $F_B$ and $F_{\bar B}$ 
have the same sets of critical values. 
It follows from the  above and from Corollary~\ref{Wn:E} that for every $[B]\in E_M$ 
the pseudo-balls $B\in[B]$ yield at most $t([B])$ non-zero special values of the pencil $f(x,y)-ty^M=0$. 
Since every $B\in T(f)$ for which $q(B)=M$ belongs to some $[B]\in E_M$ the Claim follows. 

\medskip
\textbf{Claim~2.}  
$$ \sum_{[B]\in E_M} t([B]) \leq \mbox{number of leaves of $E(f)$}\;.
$$

\textit{Proof.}
Let $E'(f)$ be the sub-tree of $E(f)$ consisting of: elements of $E_M$, 
their children and their ancestors.  
For every  $[B_1]$, $[B_2]\in E_M$ $[B_1]$ is not an ancestor of $[B_2]$. Otherwise by 
Lemma~\ref{L3} we would have $q([B_1])<q([B_2])$.  
Hence, the leaves of $E'(f)$ are exactly the children of elements of $E_M$.
It follows that the tree $E'(f)$ has $\sum_{[B]\in M} t([B])$ leaves. 
Since any sub-tree has no more leaves than the original tree the Claim is proved. 

\medskip
To finish the proof it is enough to notice that the number of leaves of $E(f)$ is equal to 
the number of irreducible components of the curve $f(x,y)=0$ counted without multiplicities. 

\section{Relation with known results}
Let $f(x,y)\in\bC\{x,y\}$ be a singular irreducible power series. Then the Eggers tree $E(f)$ has only 
one leaf, hence $E(f)$ is a chain. By Lemma~\ref{L:E} for every $B\in T(f)$ of finite height 
we have $F_B(z)=C(z^n-b)^m$ for some integers $n>1$, $m\geq0$, and some nonzero $C$ and $b$. 
Every polynomial $F_B$ has exactly one nonzero critical value $F_B(0)$.  
Using Theorems~\ref{Th:main}, \ref{T2} and~\ref{T3} we get the following characterization of special 
values of the pencil $f(x,y)-ty^M=0$: 

\begin{itemize}
\item the pencil $f(x,y)-ty^M=0$ has a special value $\tau\neq0$ if and only if there exists $B\in T(f)$ such that 
         $q(B)=M$, in this case there is only one nonzero special value $\tau=F_B(0)$.
\item the pencil $f(x,y)-ty^M=0$ has a special value $\tau=0$ if and only if there exists $B\in T(f)$ such that 
         $q(B)>M$.
\end{itemize}

The numbers $q(B)$, for $B\in T(f)$ of finite height, are so called polar invariants of $f=0$ relative to $y=0$. 
They can be defined in a coordinate free way and computed using invariants of singularity of $f=0$ and intersection  
multiplicity of $f=0$ and $y=0$ (see \cite{GB-P}).
The above characterization in the language of polar invariants is given in \cite{GB-P}, Proposition~1.1.

\medskip
In \cite{L-M} (see also \cite{L-M-P}) the authors considered the pencil $f(x,y)-x^N=0$. They distinguished 
one of edges of the Newton diagram of $f$. Then they defined the complex polynomial  $w_T$ 
associated with this edge and proved that critical values of $w_T$ are special values 
of the pencil $f(x,y)-x^N=0$. They also estimated in Theorem~3.1 the total number of special values. 

If if we interchange the r\^ole of variables $x$ and $y$ then $w_T$ is equal to one of polynomials $F_B$, 
where $q(B)=N$. Observe the analogy between Theorem~\ref{T2} of this paper applied to $F_B$ and 
the first part of the quoted result. 

\end{document}